\documentclass[reqno,12pt]{amsart}

\usepackage{epsf}
\usepackage{graphicx}
\usepackage{amssymb}
\usepackage{amsmath}

\date{}

\theoremstyle{plain}
\newtheorem{theorem}{Theorem}

\newtheorem{proposition}{Proposition}

\newtheorem*{question}{Question}

\theoremstyle{definition}
\newtheorem{definition}{Definition}

\theoremstyle{remark}

\newtheorem*{remark}{Remark}

\def\N{{\mathbb N}}
\def\Z{{\mathbb Z}}

\def\R{{\mathbb R}}

\DeclareMathOperator{\tw}{t}
\DeclareMathOperator{\h}{h}

\title{On the stabilisation height of fibre surfaces in $S^3$}

\author{Sebastian Baader \and Filip Misev}

\begin{document}

\begin{abstract} The stabilisation height of a fibre surface in the 3-sphere is the minimal number of Hopf plumbing operations needed to attain a stable fibre surface from the initial surface. We show that families of fibre surfaces related by iterated Stallings twists have unbounded stabilisation height.
\end{abstract}

\maketitle

\section{Introduction}

A famous theorem by Giroux and Goodman states that every fibre surface in $S^3$ admits a common stabilisation with the standard disc~\cite{GG}. More precisely, every fibre surface $\Sigma \subset S^3$ can be obtained from the embedded standard disc by a finite sequence of Hopf plumbing operations, followed by a finite number of Hopf deplumbing operations. The necessity of the deplumbing operation was known for a long time, thanks to examples of Melvin and Morton~\cite{MM}. In this note, we show that the number of deplumbing operations is unbounded, even for fibre surfaces of genus zero. We define the stabilisation height of a fibre surface $\Sigma \subset S^3$ as
$$\text{h}(\Sigma)=\min_S \{b_1(S)-b_1(\Sigma)\},$$
where $b_1$ denotes the first Betti number and the minimum is taken over all common stabilisations $S \subset S^3$ of $\Sigma$ and the disc. The stabilisation height measures the number of Hopf plumbing operations needed to pass from $\Sigma$ to a minimal common stabilisation $S$ with the disc, since Hopf plumbing increases the first Betti number of a surface by one. 

\begin{theorem} \label{mainthm} Let $\Sigma_n \subset S^3$ be a family of fibre surfaces of the same topological type $\Sigma$ whose monodromies $\varphi_n: \Sigma \to \Sigma$ differ by a power of a Dehn twist along an essential simple closed curve $c \subset \Sigma$:
$$\varphi_n=\varphi_0 T_c^n.$$
Then
$$\lim_{|n| \to \infty} \h(\Sigma_n)=+\infty.$$
\end{theorem}

An easy way of producing such families is by applying iterated Stal\-lings twists to a fixed fibre surface. A detailed example of this sort, as well as a definition of a Stallings twist, are presented in the next section. We do not know whether all families of monodromies related by powers of a Dehn twist arise from an iterated Stallings twist. In any case, not all fibre surfaces support Stallings twists. Indeed, a Stallings twist requires the existence of an essential, unknotted, 0-framed simple closed curve in the fibre surface. In particular, quasipositive surfaces do not admit Stallings twists.

\begin{question} Does there exist a familiy of quasipositive fibre surfaces $\Sigma_n \subset S^3$ of the same topological type with
$$\lim_{n \to \infty} \h(\Sigma_n)=+\infty?$$
\end{question}

The existence of quasipositive fibre surfaces $\Sigma \subset S^3$ (or equivalently, pages of open books supporting the standard contact structure on $S^3$) with $\h(\Sigma) \geq 1$ was recently established by Baker et al.~\cite{BEV}, Wand~\cite{Wa} and, for the genus zero case, Etnyre-Li~\cite{EL}.

The proof of Theorem~\ref{mainthm} makes use of the stable commutator length of the monodromy. The key observation is that a Hopf plumbing operation corresponds to composing the monodromy with a Dehn twist, which cannot change the commutator length too much. On the other hand, a large power of a Dehn twist has a large commutator length, thanks to a result by Endo-Kotschick and Korkmaz~\cite{EK,Ko}. A detailed proof is given in Section~3. As a preparation, we discuss the Dehn twist length of the monodromy and present a simple family of fibre surfaces of genus zero illustrating Theorem~\ref{mainthm}.

\section{The twist length of the monodromy}

The monodromy of an iterated Hopf plumbing involving $n$ Hopf bands is a product of $n$ Dehn twists. Hopf plumbing decreases the Euler characteristic of a surface by one. A surface of genus $g$ with one boundary component has Euler characteristic $1-2g$. Therefore, if a fibred knot can be obtained from the unknot by Hopf plumbing, its monodromy can be written as a product of $2g$ Dehn twists. The following proposition (taken from the second author's PhD thesis~\cite{Mi}) shows that the number $2g$ is in fact minimal.

\begin{proposition} \label{prop:tlength}
Let $K$ be a fibred knot of genus $g$ with monodromy $\varphi$. Then any representation of $\varphi$ as a product of Dehn twists involves at least $2g$ distinct factors.
\end{proposition}

The proof of Proposition~\ref{prop:tlength} has three main ingredients. Firstly, the Alexander polynomial $\Delta_K(t)\in\Z[t]$ of a fibred knot $K$ with fibre surface $\Sigma$ equals the characteristic polynomial of the homological action $\varphi_* : H_1(\Sigma,\R)\to H_1(\Sigma,\R)$ of the monodromy $\varphi$, that is,
\[ \Delta_K(t)=\det(t \text{id} - \varphi_*). \]
Secondly, if $K$ is a knot, $\Delta_K(1)=1$. This does not hold for links of more than one component: in fact, $\Delta_L(1)=0$ if $L$ is a link with at least two components. Thirdly, the homological action of a Dehn twist $T$ about a simple closed curve $\gamma$ in $\Sigma$ can be described as follows. Let $\alpha$ be any simple closed curve in $\Sigma$ and denote by $a=[\alpha]$ and $c=[\gamma]$ the classes of $\alpha$ and $\gamma$ in $H_1(\Sigma,\R)$. Then we have (see~\cite{FM})
\[ T_*(a)=a+i(a,c)c, \]
where $i(.\, ,.)$ denotes the intersection pairing.

\begin{proof}[Proof of Proposition~\ref{prop:tlength}.]
Suppose to the contrary that $\varphi$ could be written as a product of Dehn twists with $n<2g$ distinct factors $T_1,\ldots,T_n$, where $T_i=T_{c_i}$ denotes a Dehn twist about a simple closed curve $c_i$ in~$\Sigma$. Let
\[ V=\langle [c_1],\ldots,[c_n]\rangle < H_1(\Sigma,\R) \]
be the subspace of $H_1(\Sigma,\R)$ spanned by the classes of the $c_i$. Consider the orthogonal complement of $V$ in $H_1(\Sigma,\R)$ with respect to the intersection form,
\[ V^\perp = \{ x\in H_1(\Sigma,\R)\;|\; i(x,y)=0\quad\forall y\in V\}. \]
Since $\dim V\leq n<2g=\dim H_1(\Sigma,\R)$ and $i$ is non-degenerate, we have $\dim V^\perp\geq 2g-n>0$, so there exists a non-zero vector $v\in V^\perp$. We claim that $v$ is an eigenvector of $\varphi_*$ for the eigenvalue $1$. Indeed, we may write $v$ as a finite linear combination $v=\lambda_1a_1+\ldots+\lambda_ra_r$, with $\lambda_k\in\R$ and where $a_1,\ldots,a_r\in H_1(\Sigma,\R)$ can be represented by simple closed curves $\alpha_1,\ldots,\alpha_r$ in $\Sigma$. For every fixed $k\in\{1,\ldots,r\}$, we have
\begin{eqnarray}
T_k(v) &=& \sum_{j=1}^r \lambda_jT_k(a_j) = \sum_{j=1}^r \lambda_j (a_j+i(a_j,c_k)c_k) \nonumber \\
       &=& \sum_{j=1}^r \lambda_ja_j + i(\sum_{j=1}^r \lambda_ja_j, c_k)c_k = v + i(v,c_k)c_k=v+0\cdot c_k, \nonumber
\end{eqnarray}
hence $\varphi_*(v)=v$, as claimed. But then, $1=\Delta_K(1)=\chi_{\varphi_*}(1)=0$, a contradiction.
\end{proof}

\begin{definition}
Let $\Sigma$ be an abstract surface with boundary and let $f$ be a mapping class of $\Sigma$ fixing the boundary pointwise. We define the \emph{twist length} $\tw(f)$ to be the minimal number of factors in a representation of~$f$ as a product of (positive and negative) Dehn twists:
\[ \tw(f)=\min\{k\in\N\ | \ f=T_1\cdots T_k,\quad T_i\ \text{Dehn twist}\} \]
\end{definition}

\begin{remark}
The twist length $\tw(f)$ is the word length of $f$ as an element of the mapping class group $\text{Mod}(\Sigma,\partial\Sigma)$ with respect to the generating set given by all Dehn twists along essential simple closed curves in $\Sigma$.
\end{remark}

\begin{remark}
In terms of twist length, Proposition~\ref{prop:tlength} implies $\tw(\varphi)\geq 2g$ whenever $\varphi$ is the monodromy of a fibred knot of genus $g$. The examples in Subsection \ref{subsec:example} below show that the twist length can in fact be arbitrarily large for monodromies of fibred links of fixed genus.
\end{remark}

\begin{figure}[h]
\begin{center}
\includegraphics[scale=1.4]{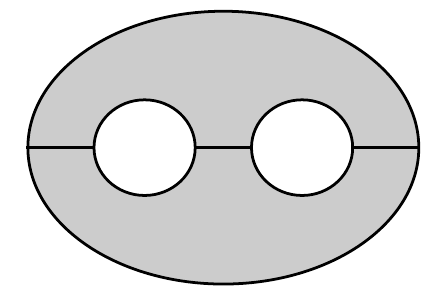}
\put(-170,30){$a$}
\put(-122,46){$b$}
\put(-57,46){$c$}
\put(-160,52){$\gamma_1$}
\put(-93,52){$\gamma_2$}
\put(-27,52){$\gamma_3$}
\put(-92,20){$\Sigma$}
\caption{A pair of pants, its three boundary curves $a,b,c$ and the three non-separating arcs $\gamma_1,\gamma_2,\gamma_3$. \label{fig:pants}}
\end{center}
\end{figure}
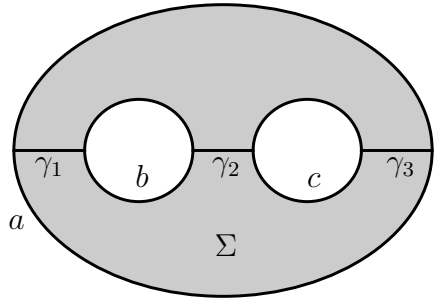

\subsection{Stallings twist}
Let $\Sigma\subset S^3$ be a fibre surface with monodromy~$\varphi$ and let $c \subset \Sigma$ be a $0$-framed simple closed curve which is unknotted in~$S^3$. A \emph{Stallings twist} along $c$ consists of a $\pm 1$ Dehn surgery along~$c^+$, a curve obtained by pushing $c$ slightly off $\Sigma$ in the positive normal direction. Stallings realised that this operation turns $\Sigma$ into another fibre surface $\Sigma'$ of the same topological type, whose monodromy is the map $\varphi$ composed with $T_c^\pm$ (compare~\cite{St}).

By Giroux and Goodman's theorem~\cite{GG}, the effect of a Stallings twist on $\Sigma$ can equally be achieved by plumbing and deplumbing Hopf bands. In particular cases (including the examples in Subsection~\ref{subsec:example} below), this can be seen explicitly, as remarked by Melvin and Morton (see Lemma, Figure~7 and Figure~8 in~\cite{MM}).

\subsection{Example} \label{subsec:example}
Let $\Sigma$ be an oriented pair of pants and denote $a,b,c$ its three boundary curves. For every given integer $n$, consider the mapping class
\[ \varphi_n=T_aT_b^{-1}T_c^n, \]
where $T_a,T_b,T_c$ are right-handed Dehn twists about (curves parallel to) $a,b,c$. Note that these Dehn twists commute since we can choose the corresponding twist curves to be disjoint. The following holds:
\begin{enumerate}
 \item $\tw(\varphi_n)=n+2$,
 \item there is a fibre surface $\Sigma_n\subset S^3$ whose monodromy is $\varphi_n$,
 \item $\lim\limits_{|n|\to\infty} \h(\Sigma_n) = +\infty$.
\end{enumerate}
The first statement follows from the fact that $a,b,c$ represent the only isotopy classes of simple closed curves in $\Sigma$, whence the mapping class group of $\Sigma$ is the free abelian group generated by $T_a$, $T_b$, $T_c$; compare \cite[Section~3.6.4]{FM}.

\begin{figure}[h]
\begin{center}
\includegraphics[scale=1.3]{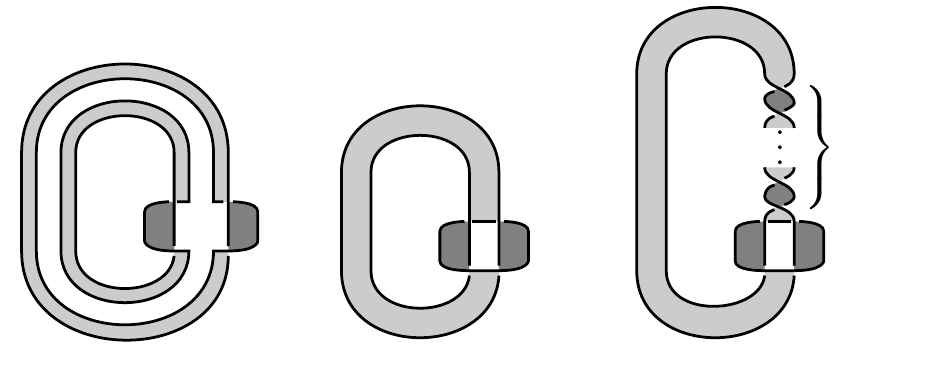}
\put(-305,75){$H_+$}
\put(-260,75){$H_-$}
\put(-225,110){$\Sigma_0=H_+\# H_-$}
\put(-45,80){\begin{minipage}[c]{2cm} \centering $n$ full\\ twists \end{minipage}}
\put(-45,10){$\Sigma_n$}
\put(-245,50){$\to$}
\put(-140,45){$\to$}
\put(-185,70){$a$}
\put(-160,70){$b$}
\put(-197,45){$c$}
\caption{Stallings twist along a curve $c$ on a connected sum of a positive and a negative Hopf band.\label{fig:Sigma}}
\end{center}
\end{figure}
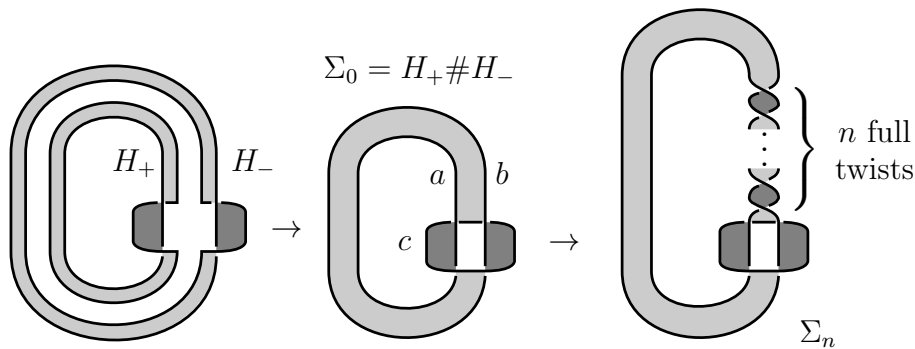

For the second statement, let $H_+$, $H_-$ be a positive Hopf band with core curve $\alpha$ and a negative Hopf band with core curve $\beta$, respectively. The connected sum of $H_+$ and $H_-$ is a fibre surface $\Sigma_0\subset S^3$ homeomorphic to~$\Sigma$. Up to a permutation of $a,b,c$, we may assume that the homeomorphism sends $(\alpha,\beta)$ to $(a,b)$, so that the monodromy of~$\Sigma_0$ is given by $\varphi_0=T_aT_b^{-1}$. The curve $c \subset \Sigma$ corresponds to an unknotted $0$-framed curve in $\Sigma_0$, along which we can perform an $n$-fold Stallings twist. This turns $\Sigma_0$ into the fibre surface $\Sigma_n$ depicted in Figure~\ref{fig:Sigma} on the right, whose monodromy is $\varphi_n$.

The third statement is an application of Theorem~\ref{mainthm}. It implies in particular that $\Sigma_n$ cannot be an iterated Hopf plumbing on the standard disc for large $n$. This can also be seen directly as follows. Suppose that $\Sigma_n$ were an iterated plumbing of Hopf bands. Undoing the last Hopf plumbing amounts to cutting the surface $\Sigma_n$ along a properly embedded non-separating arc, and the resulting surface is itself a fibre surface. However, $\Sigma_n$ is homeomorphic to a pair of pants, which contains exactly three non-separating proper arcs up to isotopy (compare \cite[Proposition~2.2]{FM}). Let $\gamma_1,\gamma_2,\gamma_3$ be the arcs connecting the boundary curves $a,b$, respectively $b,c$ and $a,c$, as shown in Figure~\ref{fig:pants}. Cutting $\Sigma_n$ along any of the arcs $\gamma_1,\gamma_2,\gamma_3$ results in an unknotted annulus with $0$, $n+1$, $n-1$ full twists, respectively. Since the only fibred annuli are the positive and the negative Hopf band, we conclude that $\Sigma_n$ cannot be a Hopf plumbing for $|n|\geq 3$.

\section{The stable commutator length of the monodromy}

The goal of this section is to prove Theorem~\ref{mainthm} by estimating the effect of an iterated Dehn twist on the stable commutator length of the monodromy. Let $G$ be a perfect group, i.e. a group in which every element is a finite product of commutators. The commutator length $\text{cl}(g)$ of an element $g \in G$ is defined as the minimal number of commutators required in a factorisation of $g$ into commutators. The \emph{stable commutator length} is defined as the following limit, whose existence follows from the sub-additivity of the function $n \mapsto \text{cl}(g^n)$:
$$\text{scl}(g)=\lim_{n \to \infty} \frac{1}{n} \text{cl}(g^n).$$
The interested reader is referred to Calegari's monograph~\cite{Ca} for more background on the stable commutator length.

We will be concerned with the stable commutator length on mapping class groups of orientable surfaces of genus greater than two, a famous family of perfect groups. The main ingredient in our proof of Theorem~\ref{mainthm} is an estimate for the stable commutator length of a Dehn twist $T_c$ along an essential simple closed curve $c \subset \Sigma$, where $\Sigma$ is a closed orientable surface of genus $g \geq 3$:
$$\text{scl}(T_c) \geq \frac{1}{18g-6}.$$
This was first proved by Endo-Kotschick for Dehn twists along separating curves~\cite{EK}, then by Korkmaz for all essential curves~\cite{Ko}.

\begin{proof}[Proof of Theorem~\ref{mainthm}]
Let $\Sigma_n \subset S^3$ be a family of fibre surfaces of the same topological type $\Sigma$ whose monodromies differ by a power of a Dehn twist along an essential simple closed curve $c \subset \Sigma$:
$$\varphi_n=\varphi_0 T_c^n.$$
Let $S_n \subset S^3$ be a common stabilisation of the standard disc and the fibre surface $\Sigma_n \subset S^3$. By plumbing at most six additional Hopf bands to $\Sigma_n$, we can make sure that 
\begin{enumerate}
\item[(i)] the genus of $S_n$ is at least three,
\item[(ii)] the complementary component(s) of the curve $c \subset S_n$ have genus at least one.
\end{enumerate}
Let $k \in \N$ be the number of Hopf plumbings needed to get from $\Sigma_n$ to $S_n$. The corresponding monodromy $\overline{\varphi_n}:S_n \to S_n$ can be written as a product of $b_1$ Dehn twists, where $b_1=b_1(S_n)$ is the first Betti number of the surface $S_n$. As a consequence, the stable commutator length of $\overline{\varphi_n}$ is bounded from above by a constant $C(b_1)$ that depends only on~$b_1$:
\begin{equation} \label{scl}
\text{scl}(\overline{\varphi_n}) \leq C(b_1).
\end{equation}
Indeed, there is an upper bound on the commutator length of Dehn twists for a surface with finite first Betti number, since there are only finitely many conjugacy classes of Dehn twists in the mapping class group of a fixed surface.

\medskip
The remainder of the proof consists in showing
$$\lim_{|n| \to \infty} \text{scl}(\overline{\varphi_n})=+\infty,$$
provided that $b_1(S_n)$ is bounded while $|n|$ tends to infinity. In order to estimate the stable commutator length $\text{scl}(\overline{\varphi_n})$, we cap off the surface $S_n$ by discs (abstractly, not in $S^3$) and extend the monodromy $\overline{\varphi_n}$ by the identity on these discs. We thus obtain a diffeomorphism
$$\widetilde{\varphi_n}: \widetilde{S_n} \to \widetilde{S_n}$$
of a closed surface $\widetilde{S_n}$ of genus at least three whose stable commutator length satisfies
$$\text{scl}(\overline{\varphi_n}) \geq \text{scl}(\widetilde{\varphi_n}).$$
By construction, the map $\widetilde{\varphi_n}$ can be expressed as
$$\widetilde{\varphi_n}=T_k T_{k-1} \ldots T_1 \varphi_0 T_c^n,$$
where $T_1,T_2,\ldots,T_k$ are Dehn twists along the core curves of the $k$ Hopf bands plumbed to $\Sigma_n$. An elementary calculation (compare Section~2.7.4 on free products in~\cite{Ca}) shows
$$\text{scl}(gh) \geq \text{scl}(g)+\text{scl}(h)-1,$$
for arbitrary elements $g,h$ of a perfect group. This yields
\begin{align*}\text{scl}(\widetilde{\varphi_n}) &\geq \text{scl}(T_k T_{k-1} \ldots T_1 \varphi_0)+\text{scl}(T_c^n)-1 \\
&\geq \text{scl}(T_k T_{k-1} \ldots T_1)+\text{scl}(\varphi_0)+\text{scl}(T_c^n)-2 \\
&\geq \ldots \\
&\geq \sum_{i=1}^k \text{scl}(T_i)+\text{scl}(\varphi_0)+\text{scl}(T_c^n)-(k+1).
\end{align*}
At last, we apply Korkmaz' result (Theorem~2.1 in~\cite{Ko}) to the curve $c \subset \widetilde{S_n}$, which is non-trivial, since its complementary regions have strictly positive genus (compare~(ii)): $\text{scl}(T_c)>0$. This implies
$$\lim_{|n| \to \infty} \text{scl}(T_c^n)=\lim_{|n| \to \infty} |n| \, \text{scl}(T_c)=+\infty,$$
hence
$$\lim_{|n| \to \infty} \text{scl}(\overline{\varphi_n})=+\infty,$$
provided that $k$ (equivalently $b_1(S_n)$) is bounded while $|n|$ tends to infinity. Here again we use the fact that there is an upper bound on the commutator length of Dehn twists for a surface with finite first Betti number. This is in contradiction with~(\ref{scl}). We thus conclude
$$\displaystyle{\lim_{|n| \to \infty} b_1(S_n)=+\infty},$$
which is precisely the statement of Theorem~\ref{mainthm}: the stabilisation height of $S_n$ tends to infinity as $|n|$ does. 
\end{proof}

\bigskip
\noindent
Universit\"at Bern, Sidlerstrasse 5, CH-3012 Bern, Switzerland

\bigskip
\noindent
\texttt{sebastian.baader@math.unibe.ch}

\medskip
\noindent
\texttt{filip.misev@math.unibe.ch}


\begin{thebibliography}{99}

\bibitem{BEV}
     K. Baker, J. B. Etnyre, J. Van Horn-Morris: \emph{Cabling, contact structures and mapping class monoids},
J. Differential Geom.~\textbf{90} (2012), no.~1, 1--80. 

\bibitem{Ca}
     D. Calegari: \emph{scl}, MSJ Memoirs, 20. Mathematical Society of Japan, Tokyo, 2009.

\bibitem{EK}
     H. Endo, D. Kotschick: \emph{Bounded cohomology and non-uniform perfection of mapping class groups}, Invent. Math.~\textbf{144} (2001), no.~1, 169--175.

\bibitem{EL}
     J. Etnyre, Y. Li: \emph{The arc complex and contact geometry: nondestabilizable planar open book decompositions of the tight contact 3-sphere}, Int. Math. Res. Not. IMRN 2015, no.~5, 1401--1420.

\bibitem{FM}
     B. Farb, D. Margalit: \emph{A primer on mapping class groups}, Princeton University Press, 2012.


\bibitem{GG}
     \'E. Giroux, N. Goodman: \emph{On the stable equivalence of open books in three-manifolds}, Geom. Topol.~\textbf{10} (2006), 97--114.


\bibitem{Ko}
     M. Korkmaz: \emph{Stable commutator length of a Dehn twist}, Michigan Math. J.~\textbf{52} (2004), no.~1, 23--31.

\bibitem{MM}
     P. M. Melvin, H. R. Morton: \emph{Fibred knots of genus 2 formed by plumbing Hopf bands}, J.~London Math. Soc. (2)~\textbf{34} (1986), no.~1, 159--168.

\bibitem{Mi}
     F. Misev: \emph{On the plumbing structure of fibre surfaces}, PhD thesis, University of Bern, 2016.

\bibitem{St}
     J. R. Stallings: \emph{Constructions of fibred knots and links}, Algebraic and geometric topology, Proc. Sympos. Pure Math. XXXII (1978), Part~2, 55--60, Amer. Math. Soc., Providence, R.I. 

\bibitem{Wa}
     A. Wand: \emph{Factorizations of diffeomorphisms of compact surfaces with boundary}, Geom. Topol.~\textbf{19} (2015), no.~5, 2407--2464.

\end{thebibliography}
\end{document}